\documentclass[11pt]{article}
\usepackage{amsmath}
\usepackage{amssymb}
\usepackage{amsthm}
\usepackage{bookmark}
\usepackage{geometry}
\usepackage{hyperref}
\newtheorem{theorem}{Theorem}
\newtheorem{corollary}[theorem]{Corollary}
\newtheorem{lemma}[theorem]{Lemma}
\newtheorem{proposition}[theorem]{Proposition}
\theoremstyle{definition}
\newtheorem{example}[theorem]{Example}

\DeclareMathOperator{\N}{N}
\DeclareMathOperator{\Tr}{Tr}

\title{On a Class of Permutation Polynomials and Their Inverses}
\author{Ruikai Chen\textsuperscript{1,2}\and Sihem Mesnager\textsuperscript{1,2,3}}
\date{\small\textsuperscript{1}Department of Mathematics, University of Paris VIII, F-93526 Saint-Denis\\\textsuperscript{2}Laboratory Analysis, Geometry and Applications, LAGA, University Sorbonne Paris Nord, CNRS, UMR 7539, F-93430, Villetaneuse, France\\\textsuperscript{3}Telecom Paris, Polytechnic institute of Paris, 91120 Palaiseau, France\\Emails: \href{mailto:chen.rk@outlook.com}{chen.rk@outlook.com}\quad\href{mailto:smesnager@univ-paris8.fr}{smesnager@univ-paris8.fr}}

\begin{document}

\maketitle

\noindent\textbf{Abstract.} We introduce a class of permutation polynomial over $\mathbb F_{q^n}$ that can be written in the form $\frac{L(x)}{x^{q+1}}$ or $\frac{L(x^{q+1})}x$ for some $q$-linear polynomial $L$ over $\mathbb F_{q^n}$. Specifically, we present those permutation polynomials explicitly as well as their inverses. In addition, more permutation polynomials can be derived in a more general form.\\
\textbf{Keywords.} permutation polynomial, compositional inverse, polynomial, finite field

\section{Introduction}

Let $q$ be a prime power and $\mathbb F_q$ denote the finite field of order $q$. A permutation polynomial over $\mathbb F_q$ is a polynomial that acts as a permutation on $\mathbb F_q$ in the usual way (see \cite[Chapter 7]{lidl1997}). Permutation polynomials play an essential role in the theory of finite fields as well as applications in coding theory, cryptography, combinatorics, etc.

If $f$ is a permutation polynomial of $\mathbb F_q$, then there exists a polynomial over $\mathbb F_q$, denoted $f^{-1}$, such that $f(f^{-1}(x))\equiv f^{-1}(f(x))\equiv x\pmod{x^q-x}$. We call it the (compositional) inverse of $f$, which is unique as a permutation. In general, given a map from a finite field to itself, the corresponding polynomial can be retrieved using the Lagrange interpolation formula. Particularly, one can refer to \cite{muratovic2007note} to compute the coefficients of the inverse of a permutation polynomial and to \cite{wang2009inverse} for a formula of the inverse of a permutation polynomial of the form $x^rh(x^s)$. However, it is still challenging to find a permutation polynomial and its inverse both of which have simple forms. In addition to monomials, Dickson polynomials and linearized polynomials, there are only a limited number of known examples, to name a few, including a class of bilinear permutation polynomials in \cite{coulter2002compositional}, permutation polynomials from some piecewise constructions in \cite{zheng2015piecewise}, and some permutation polynomials involving linearized polynomials in \cite{tuxanidy2014inverses,tuxanidy2017compositional}. See also a summary of inverses of permutation polynomials of small degree in \cite{zheng2019inverses}.

In recent years, much work has been done on permutation polynomials of the form $f(x)=xh(x^{q-1})$ over $\mathbb F_{q^3}$. Some examples of trinomials are listed as follows:
\begin{itemize}
\item $2x^\frac{q^3+q}2+x^q+x$ (\cite[Table 1]{cao2011new});
\item $x^\frac{q^3+q}2+x^\frac{q^2+1}2-x$ (\cite[Theorem 4.1]{ma2017some});
\item $ax^{q^3-q^2+q}+bx^{q^2-q+1}+x$ (\cite[Theorem 1]{xie2023two});
\item $x^{q^2+q-1}+ax^{q^2-q+1}+bx$ (\cite[Theorem 3.1]{bartoli2020permutation});
\item $x^{q^2+q-1}+ax^{q^2}+bx$ (\cite[Theorem 3.4]{bartoli2020permutation});
\item $x^{q^2-q+1}+ax^{q^2}+bx$ (\cite[Theorem 4.1]{gupta2022new}, \cite[Theorem 2]{xie2023two})
\end{itemize}
for odd $q$, and
\begin{itemize}
\item $x^{q^2-q+1}+ax^{q^2}+bx$ (\cite[Theorem 1]{pang2021permutation});
\item $ax^{q^2-q+1}+bx^{q^3-q^2+q}+x$ (\cite[Theorem 4]{pang2021permutation});
\item $x^{q^2+q-1}+ax^{q^2-q+1}+x$ (\cite[Theorem 3.2]{gupta2022new});
\item $x^{q^2+q-1}+ax^{q^3-q^2+q}+x$ (\cite[Theorem 3.4]{gupta2022new})
\end{itemize}
for even $q$, where $a$ and $b$ are coefficients in $\mathbb F_{q^3}$ satisfying some conditions therein. As we will see, some of the above polynomials are included in a more general class of permutation polynomials.

In this paper, we investigate a class of permutation polynomials that can be written in the form $\frac{L(x)}{x^{q+1}}$ or $\frac{L(x^{q+1})}x$ for some $q$-linear polynomial $L$ over $\mathbb F_{q^n}$ (i.e., polynomials that induce linear endomorphisms of $\mathbb F_{q^n}$ over $\mathbb F_q$). Here for convenience $\frac1x$ (or $x^{-1}$) does not denote a fraction but $x^{q^n-2}$ in the polynomial ring $\mathbb F_{q^n}[x]$. Also, let $x^{q^n}-x=0$ by abuse of notation since all those polynomials are treated as maps from $\mathbb F_{q^n}$ to itself. Let $\Tr$ and $\N$ be functions from $\mathbb F_{q^n}$ to $\mathbb F_q$ defined as
\[\Tr(x)=\sum_{i=0}^{n-1}x^{q^i}\quad\text{and}\quad\N(x)=\prod_{i=0}^{n-1}x^{q^i}.\]
In Section 2, we first establish some properties of those permutation polynomials and then present instances with inverses in explicit forms. A more general form of permutation polynomials will be discussed in Section 3.

\section{Two Kinds of Permutation Polynomials with Inverses}

Let $L$ be a $q$-linear polynomial over $\mathbb F_{q^n}$ as 
\[L(x)=\sum_{i=0}^na_ix^{q^i}.\]
Then we call the polynomial $L^\prime$ defined by
\[L^\prime(x)=\sum_{i=0}^{n-1}(a_ix)^{q^{n-i}}\]
the transpose of $L$, in the sense of linear endomorphisms of $\mathbb F_{q^n}/\mathbb F_q$. In fact, given a basis $\alpha_1,\dots,\alpha_n$ of $\mathbb F_{q^n}/\mathbb F_q$ and its dual basis $\beta_1,\dots,\beta_n$ (with respect to the usual bilinear form), the matrix with $(i,j)$ entry $\Tr(\beta_iL(\alpha_j))$ represents $L$ with respect to the basis $\alpha_1,\dots,\alpha_n$. Its transpose represents $L^\prime$ with respect to the dual basis, since $\Tr(\beta_jL(\alpha_i))=\Tr(\alpha_iL^\prime(\beta_j))$. This leads to the following result.

\begin{theorem}
If $\frac{L(x)}{x^{q+1}}$ is a permutation polynomial of $\mathbb F_{q^n}$, then so is $\frac{L^\prime(x^{q+1})}x$. The converse is true provided that $n$ is odd.
\end{theorem}
\begin{proof}
Suppose $\frac{L(x)}{x^{q+1}}$ permutes $\mathbb F_{q^n}$. Then given any $u\in\mathbb F_{q^n}^*$ with $u\ne1$,
\[\frac{L(x)}{x^{q+1}}-\frac{L(ux)}{(ux)^{q+1}}\]
has no root in $\mathbb F_{q^n}^*$, and neither does
\[u^{q+1}L(x)-L(ux),\]
whose transpose is
\[L^\prime(u^{q+1}x)-uL^\prime(x).\]
As the transpose preserves the rank of a linear map, the polynomial
\[\frac{L^\prime((ux)^{q+1})}{ux}-\frac{L^\prime(x^{q+1})}x\]
has no root in $\mathbb F_{q^n}^*$, which means $\frac{L^\prime(x^{q+1})}x$ is injective on $\mathbb F_{q^n}$, and hence a permutation.

Suppose $n$ is odd. For the converse, it suffices to prove that if $L^\prime((ux)^{q+1})-uL^\prime(x^{q+1})$ has no root in $\mathbb F_{q^n}^*$, then neither does $L^\prime(u^{q+1}x)-uL^\prime(x)$. Note that
\[q^n-1=(q+1-1)^n-1\equiv-2\pmod{q+1},\]
so $\gcd(q+1,q^n-1)=\gcd(q+1,2)$. If $q$ is even, then $\gcd(q+1,q^n-1)=1$ and the claim is obvious. If $q$ is odd, then $\gcd(q+1,q^n-1)=2$, and there exists an element $\alpha\in\mathbb F_q$ that is a non-square in $\mathbb F_{q^n}$. Assume $L^\prime(u^{q+1}x_0)-uL^\prime(x_0)=0$ for some $x_0\in\mathbb F_{q^n}^*$. Then the same holds for $\alpha x_0$, as can be easily seen. Therefore, one of $x_0$ and $\alpha x_0$ is a square in $\mathbb F_{q^n}^*$ and gives rise to a root of $L^\prime((ux)^{q+1})-uL^\prime(x^{q+1})$ in $\mathbb F_{q^n}^*$. This shows that $\frac{L(x)}{x^{q+1}}$ is indeed injective.
\end{proof}

\begin{theorem}
If $\frac{L(x)}{x^{q+1}}$ is a permutation polynomial of $\mathbb F_{q^n}$, then so are $L$ and $\frac{L^{-1}(x^{q+1})}x$. The converse is true provided that $n$ is odd.
\end{theorem}
\begin{proof}
Suppose $\frac{L(x)}{x^{q+1}}$ is a permutation polynomial of $\mathbb F_{q^n}$. Clearly it maps zero to zero (and $\mathbb F_{q^n}^*$ to $\mathbb F_{q^n}^*$), so $L(x)$ has no nonzero root in $\mathbb F_{q^n}$, and is invertible on $\mathbb F_{q^n}$. For arbitrary $t\in\mathbb F_{q^n}^*$, the polynomial $\frac{L(x)}{x^{q+1}}+t^{q+1}$, as well as $L(x)+t^{q+1}x^{q+1}$, has exactly one root in $\mathbb F_{q^n}^*$. Apply $L^{-1}$ to $L(x)+t^{q+1}x^{q+1}$ and then substitute $t^{-1}x$ for $x$ to get
\[t^{-1}x+L^{-1}(x^{q+1}),\]
which also has exactly one root in $\mathbb F_{q^n}^*$. This proves the first part.

Suppose $n$ is odd, and $L$ is a permutation polynomial of $\mathbb F_{q^n}$ as well as $\frac{L^{-1}(x^{q+1})}x$. By the same argument,
\[L(x)+t^{q+1}x^{q+1}\]
has exactly one root in $\mathbb F_{q^n}^*$ for each $t\in\mathbb F_{q^n}^*$. If $q$ is even, then it is apparent as $\gcd(q+1,q^n-1)=1$. Suppose $q$ is odd. Then $\gcd(q+1,q^n-2)=2$ and there exists an element $\alpha\in\mathbb F_q$ that is a non-square in $\mathbb F_{q^n}$. Thus,
\[L(x)+\alpha t^{q+1}x^{q+1}=\alpha(L(\alpha^{-1}x)+t^{q+1}x^{q+1})\]
and
\[L(x)+(t\alpha)^{q+1}x^{q+1}\]
has the same number of roots in $\mathbb F_{q^n}^*$. The proof is complete since any element of $\mathbb F_{q^n}^*$ is either $t^{q+1}$ or $\alpha t^{q+1}$ for some $t\in\mathbb F_{q^n}^*$.
\end{proof}

As a direct result, if $L(x)=x^q+ax$ for some $a\in\mathbb F_{q^n}^*$ with $\N(-a)\ne1$, then clearly $\frac{L(x)}{x^{q+1}}=x^{-1}+ax^{-q}$ is a permutation polynomial of $\mathbb F_{q^n}$, as well as $\frac{L^{-1}(x^{q+1})}x$ by the above theorem. Furthermore, if $n$ is odd and $L(x)=x^{q^{n-1}}+ax$, then $\frac{L^{-1}(x)}{x^{q+1}}$ is a permutation polynomial of $\mathbb F_{q^n}$, by a similar argument. In fact, the latter holds for arbitrary $n$. The results are summarized as shown below.

\begin{proposition}
Let $a$ be an element in $\mathbb F_{q^n}^*$ with $\N(-a)\ne1$. If $L(x)=x^q+ax$, then $\frac{L^{-1}(x^{q+1})}x$ is a permutation polynomial of $\mathbb F_{q^n}$. If $L(x)=x^{q^{n-1}}+ax$, then $\frac{L^{-1}(x)}{x^{q+1}}$ is a permutation polynomial of $\mathbb F_{q^n}$.
\end{proposition}
\begin{proof}
The first statement has been proved. Let $L(x)=x^{q^{n-1}}+ax$. For $t\in\mathbb F_{q^n}^*$, the polynomial $1+tx+a^qt^qx^{q^2}$ has exactly one root in $\mathbb F_{q^n}^*$, for otherwise $tx_0+a^qt^qx_0^{q^2}=0$ for some $x_0\in\mathbb F_{q^n}^*$, which means $-a^q=t^{1-q}x_0^{1-q^2}$ and $\N(-a)=1$. Since
\[x+L(tx^{q+1})=x+t^{q^{n-1}}x^{q^{n-1}+1}+atx^{q+1}=x\big(1+tx+a^qt^qx^{q^2}\big)^{q^{n-1}},\]
the polynomial $L^{-1}(x)+tx^{q+1}$, as well as $x+L(tx^{q+1})$, also has exactly one root in $\mathbb F_{q^n}^*$.
\end{proof}

The inverse of a $q$-linear permutation binomial can be obtained from \cite[Theorem 2.1]{wu2013compositional} as follows.

\begin{lemma}\label{inverse}
For $a\in\mathbb F_{q^n}^*$ with $\N(-a)\ne1$ and a positive integer $k$ with $\gcd(k,n)=1$, the inverse of $x^{q^k}+ax$ is
\[\frac{\N(a)}{\N(a)+(-1)^{n-1}}\sum_{i=0}^{n-1}(-1)^ia^{-\frac{q^{k(i+1)}-1}{q^k-1}}x^{q^{ki}}.\]
\end{lemma}

\begin{example}\label{e0}
In particular, in the case $n=3$, the inverse of $x^q+ax$ with $\N(a)\ne-1$ is
\[\frac{x^{q^2}-a^{q^2}x^q+a^{q^2+q}x}{\N(a)+1}.\]
Thus
\[\frac{x^{q^2+1}-a^{q^2}x^{q^2+q}+a^{q^2+q}x^{q+1}}x=x^{q^2}-a^{q^2}x^{q^2+q-1}+a^{q^2+q}x^q\]
is a permutation polynomial of $\mathbb F_{q^3}$, and so is its $q$-th power
\[x-ax^{q^2-q+1}+a^{q^2+1}x^{q^2}.\]
Substituting $x^\frac{q^3+q}2$ for $x$ here, we get
\[x^\frac{q^3+q}2-ax+a^{q^2+1}x^\frac{q^2+1}2,\]
also a permutation polynomial of $\mathbb F_{q^3}$ in the case of odd $q$, since $\gcd(\frac{q^2+1}2,q^3-1)=1$ with $\frac{q^2+1}2(q^2-q+1)q\equiv1\pmod{q^3-1}$. Note that those permutation trinomials from \cite[Theorem 4.1]{gupta2022new}, \cite[Theorem 2]{xie2023two} and \cite[Theorem 4.1]{ma2017some}, as special cases, can be obtained from the above discussion along with Example \ref{e1} below.
\end{example}

Now, we further the study by determining the inverses of some of those permutation polynomials.

\begin{theorem}
Let $n$ be odd, $r=\sum_{i=0}^\frac{n-1}2q^{2i}$ and $s=\sum_{i=1}^\frac{n-1}2q^{2i-1}$. If $L$ is the inverse of $x^{q^{n-1}}+ax$ for some $a\in\mathbb F_{q^n}^*$ with $\N(-a)\ne1$, then the inverse of $\frac{L(x)}{x^{q+1}}$ as a permutation polynomial of $\mathbb F_{q^n}$ is $\frac{\ell(x^s)}{x^r}$, where $\ell$ is the inverse of $x^{q^{n-1}}+ax^q$.
\end{theorem}
\begin{proof}
Note first that
\[L^{-1}(x)x^s=x^{q^{n-1}+s}+ax^{1+s}=x^{rq^{n-1}}+ax^{rq}=\ell^{-1}(x^r),\]
and then
\[\ell(xL(x)^s)=L(x)^r.\]
It turns out that
\[\frac{\ell\left(\left(\frac{L(x)}{x^{q+1}}\right)^s\right)}{\left(\frac{L(x)}{x^{q+1}}\right)^r}=\frac{\ell\left(\frac{xL(x)^s}{\N(x)}\right)}{\frac{L(x)^r}{x\N(x)}}=\N(x)^{q^n-1}\frac{x\ell(xL(x)^s)}{L(x)^r}=xL(x)^{r(q^n-1)}=x.\]
\end{proof}

Note that both $\frac{L(x)}{x^{q+1}}$ and $\frac{\ell(x^s)}{x^r}$ have $n$ terms. For example, consider the case $n=3$ and let
\[L(x)=\frac{-a^qx^{q^2}+x^q+a^{q^2+q}x}{\N(a)+1},\]
the inverse of $x^{q^2}+ax$, with
\[\ell(x)=\frac{a^{q+1}x^{q^2}+x^q-a^qx}{\N(a)+1},\]
the inverse of $x^{q^2}+ax^q$. Then
\[(\N(a)+1)\frac{L(x)}{x^{q+1}}=-a^qx^{q^2-q-1}+x^{-1}+a^{q^2+q}x^{-q}\]
is a permutation polynomial of $\mathbb F_{q^3}$ with inverse
\[(\N(a)+1)\frac{\ell(x^q)}{x^{q^2+1}}=a^{q+1}x^{-q^2}+x^{-1}-a^qx^{q-q^2-1}.\]

\begin{theorem}
The inverse of $\frac{L(x^{q+1})}x$ as a permutation polynomial of $\mathbb F_{q^n}$ is $\frac{x^{-1}}{\ell(x^{-q-1})}$, if
\begin{itemize}
\item $L$ is the inverse of $x^q+ax$ for some $a\in\mathbb F_{q^n}^*$ with $\N(-a)\ne1$ and $\ell$ is the inverse of $ax^q+x$, or
\item the image of $L(x^{q+1})^{q+1}$ is contained in $\beta\mathbb F_{q^k}$ for some $\beta\in\mathbb F_{q^n}^*$ with $k$ dividing $n$, $\ell$ is a $q^k$-linear polynomial with $\ell(x)=\beta^{-1}L(\beta x)^q$, and $L(x^{q+1})$ has no root in $\mathbb F_{q^n}^*$.
\end{itemize}
In this case, $\frac{\ell(x^{q+1})}x$ is also a permutation polynomial of $\mathbb F_{q^n}$ with inverse $\frac{x^{-1}}{L(x^{-q-1})}$.
\end{theorem}
\begin{proof}
It suffices to show
\[\frac{L(x^{q+1})}x\ell\left(\left(\frac x{L(x^{q+1})}\right)^{q+1}\right)=\frac1x.\]
Assume the first condition. Then
\[\frac{L^{-1}(x)}{x^{q+1}}=\ell^{-1}\left(\frac1x\right),\]
and
\[\ell\left(\frac{L^{-1}(x)}{x^{q+1}}\right)=\frac1x.\]
Substitute $L^{-1}(x)$ for $x$ to get
\[\ell\left(\frac x{L(x)^{q+1}}\right)=\frac1{L(x)},\]
which implies
\[\ell\left(\frac{x^{q+1}}{L(x^{q+1})^{q+1}}\right)=\frac1{L(x^{q+1})},\]
and
\[\frac{L(x^{q+1})}x\ell\left(\left(\frac x{L(x^{q+1})}\right)^{q+1}\right)=\frac{L(x^{q+1})^{q^n-1}}x=\frac1x.\]

Similarly, the second condition implies
\[\begin{split}&\mathrel{\phantom{=}}\frac{L(x^{q+1})}x\ell\left(\left(\frac x{L(x^{q+1})}\right)^{q+1}\right)\\&=\frac{L(x^{q+1})}x\ell\left(\frac{\beta^{-1}x^{q+1}}{\beta^{-1}L(x^{q+1})^{q+1}}\right)\\&=\frac\beta{xL(x^{q+1})^q}\ell\left(\beta^{-1}x^{q+1}\right)\\&=\frac1xL(x^{q+1})^{q(q^n-1)}\\&=\frac1x.\end{split}\]
The results follow immediately.
\end{proof}

\begin{corollary}
Let $n$ be even with $n=2k$. Then $\frac{L(x^{q+1})}x$ is a permutation polynomials of $\mathbb F_{q^n}$ with inverse $\frac{x^{-1}}{L(x^{-q-1})^q}$, if
\begin{itemize}
\item $k$ is odd and $L(x)=(bx)^{q^{n-1}}+(bx)^{q^{k-1}}$ for $b\in\mathbb F_{q^n}^*$ such that $b^\frac{q^n-1}{q+1}\ne-1$, or
\item $k$ is even, $q$ is odd, and $L(x)=(bx)^{q^{n-1}}+(bx)^{q^{k-1}}$ for a square element $b\in\mathbb F_{q^n}^*$, or
\item $n=4$, $q$ is odd, and $L(x)=(ax^q)^{q^2}+(bx)^{q^2}+ax^q+bx$ for square elements $a,b\in\mathbb F_{q^4}^*$ such that $ab^{-q}\in\mathbb F_{q^2}$ and $\N(a)\ne\N(b)$.
\end{itemize}
\end{corollary}
\begin{proof}
In each case, the image of $L(x^{q+1})^{q+1}$ is contained in some subfield of $\mathbb F_{q^n}$, over which $L(x)^q$ induces a linear endomorphism, so it remains to show that $L(x^{q+1})$ has no root in $\mathbb F_{q^n}^*$.

For $L(x)=(bx)^{q^{n-1}}+(bx)^{q^{k-1}}$, note that any root of $L$ in $\mathbb F_{q^n}^*$ is $b^{-1}\tau$ for some $\tau\in\mathbb F_{q^n}^*$ such that $\tau^{q^k}+\tau=0$. If $k$ is odd and $b^\frac{q^n-1}{q+1}\ne-1$, then
\[(b^{-1}\tau)^\frac{q^n-1}{q+1}=(b^{-1}\tau)^{(q^k-1)\frac{q^k+1}{q+1}}=\big(-b^{1-q^k}\big)^\frac{q^k+1}{q+1}=-b^\frac{1-q^n}{q+1}\ne1.\]
If $k$ is even, $q$ is odd and $b^\frac{q^n-1}2=1$, then $(b^{-1}\tau)^\frac{q^n-1}{q+1}\ne1$, for
\[(b^{-1}\tau)^{\frac{q^n-1}{q+1}\frac{q+1}2}=(b^{-1}\tau)^\frac{q^n-1}2=\tau^{(q^k+1)\frac{q^k-1}2}=(-1)^\frac{q^k-1}2\tau^{q^k-1}=-1.\]
In both cases $b\tau^{-1}$ can not be a $(q+1)$-st power in $\mathbb F_{q^n}^*$, so $L(x^{q+1})$ has no root in $\mathbb F_{q^n}^*$.

Consider the last case. Note that $ab^{-q}x^q+x$ is a permutation polynomial of $\mathbb F_{q^4}$ as $\N(-ab^{-q})\ne1$. For $\tau\in\mathbb F_{q^4}$, if $L(b^{-1}\tau)=0$, then
\[ab^{-q}\big(\tau^{q^2}+\tau\big)^q+\tau^{q^2}+\tau=L(b^{-1}\tau)=0,\]
which means $\tau^{q^2}+\tau=0$. By the same argument, $(b^{-1}\tau)^\frac{q^n-1}{q+1}\ne1$ and $L(x^{q+1})$ has no root in $\mathbb F_{q^4}^*$.
\end{proof}

The above results show a way to obtain a permutation polynomial by indicating its inverse. In what follows, another example is given for the case $n=3$ (cf. \cite[Theorem 1]{xie2023two} and Example \ref{e0}).

\begin{proposition}
Let $q$ be odd and $a\in\mathbb F_{q^3}^*$ with $\N(a)\ne-1$. Then $x^{q^2-q+1}+2ax+a^2x^{q^3-q^2+q}$ is a permutation polynomial of $\mathbb F_{q^3}$ with inverse
\[\frac{-a^{2q+1}x^{q^2}+a^{q^2+q}x-ax^q+2a^{q+1}x^\frac{q^4+q^2}2+(1-\N(a))x^\frac{q^3+q}2+(\N(a)-1)a^qx^\frac{q^2+1}2}{(\N(a)+1)^2}.\]
\end{proposition}
\begin{proof}
Let $g(x)=x-a^{-q^2}x^\frac{q^2+1}2-ax^\frac{q^3+q}2$. Note that
\[f(x)=\frac{\big(x^{q^2}+ax^q\big)^2}{x^{q^2+q-1}}.\]
Then
\[f(x)^\frac{q^2+1}2=\frac{\big(x^{q^2}+ax^q\big)^{q^2+1}}{x^q}=x^{q^2}+ax^q+a^{q^2}x^{q^2-q+1}+a^{q^2+1}x,\]
and hence,
\[\begin{split}g(f(x))&=x^{q^2-q+1}+2ax+a^2x^{q^3-q^2+q}-a^{-q^2}\big(x^{q^2}+ax^q+a^{q^2}x^{q^2-q+1}+a^{q^2+1}x\big)\\&\mathrel{\phantom{=}-}a\big(x+a^qx^{q^2}+ax^{q^3-q^2+q}+a^{q+1}x^q\big)\\&=-a^{-q^2}(\N(a)+1)\big(x^{q^2}+ax^q\big),\end{split}\]
which is a permutation polynomial of $\mathbb F_{q^3}$ with inverse
\[\frac{-a^{2q+1}x^{q^2}+a^{q^2+q}x-ax^q}{(\N(a)+1)^2}.\]
Accordingly, the inverse of $f$ is
\[\frac{-a^{2q+1}g(x)^{q^2}+a^{q^2+q}g(x)-ag(x)^q}{(\N(a)+1)^2}.\]
Expanding this yields the desired result.
\end{proof}

\section{A More General Form of Permutation Polynomials}

Given two $q$-linear permutation polynomial $L$ and $\ell$ over $\mathbb F_{q^n}$, consider the polynomial in the form $\frac{L(x^{q+1})}{\ell(x)}$. In the subsequent part, we discuss this kind of permutation polynomials and provide some instances.

First, assume that $q$ is odd. For $t\in\mathbb F_{q^n}$, let $N_t$ be the number of roots of $\frac{L(x^{q+1})}{\ell(x)}+t$ in $\mathbb F_{q^n}$. Clearly $N_0=1$ and $N_t+1$ is the number of roots of $L(x^{q+1})+t\ell(x)$ in $\mathbb F_{q^n}$ if $t\ne0$, which can be determined using character sums. Let $\chi$ be the canonical additive character of $\mathbb F_{q^n}$. Then for $t\in\mathbb F_{q^n}^*$,
\[q^n(N_t+1)=\sum_{w\in\mathbb F_{q^n}}\sum_{u\in\mathbb F_{q^n}}\chi\left(u(L(w^{q+1})+t\ell(w))\right)=\sum_{u\in\mathbb F_{q^n}}\sum_{w\in\mathbb F_{q^n}}\chi\left(L^\prime(u)w^{q+1}+\ell^\prime(tu)w\right).\]
The inner sum
\[\sum_{w\in\mathbb F_{q^n}}\chi\left(L^\prime(u)w^{q+1}+\ell^\prime(tu)w\right)\]
can be evaluated according to \cite[Theorem 1]{coulter1998further}. In the following, we give a refined version of that for odd $n$. Let $\eta$ be the quadratic character of $\mathbb F_{q^n}$ with $\eta(0)=0$, $G$ be the corresponding Gauss sum (see \cite[Theorem 5.15]{lidl1997}), and
\[\lambda(x)=\sum_{i=0}^{n-1}(-1)^{i+1}x^{q^{2i}}.\]

\begin{lemma}
Let $q$ and $n$ be odd and $A,B\in\mathbb F_{q^n}$ with $A\ne0$. Then
\[\sum_{w\in\mathbb F_{q^n}}\chi(Aw^{q+1}+Bw)=\eta(A)G\chi\left(-\frac{\lambda(A^sB^q)^{q+1}}{4\N(A)}\right),\]
where $s=\sum_{i=1}^\frac{n-1}2q^{2i}$.
\end{lemma}
\begin{proof}
By \cite[Theorem 1]{coulter1998further}, we have
\[\sum_{w\in\mathbb F_{q^n}}\chi(Aw^{q+1}+Bw)=\eta(A)G\chi(-A\theta^{q+1}),\]
where $\theta$ is the unique element in $\mathbb F_{q^n}$ such that $A^q\theta^{q^2}+A\theta+B^q=0$. Note that
\[(A^{1+s}\theta)^{q^2}+A^{1+s}\theta+A^sB^q=A^s(A^q\theta^{q^2}+A\theta+B^q)=0.\]
Then by Lemma \ref{inverse},
\[\theta=\frac{\sum_{i=0}^{n-1}(-1)^i(-A^sB^q)^{q^{2i}}}{2A^{1+s}}=\frac{\lambda(A^sB^q)}{2A^{1+s}},\]
and
\[\theta^{q+1}=\frac{\lambda(A^sB^q)^{q+1}}{4A\N(A)},\]
as desired.
\end{proof}

Now we are ready to provide a characterization of permutation polynomials of the form $\frac{L(x^{q+1})}{\ell(x)}$.

\begin{proposition}
Let $q$ and $n$ be odd. Then $\frac{L(x^{q+1})}{\ell(x)}$ is a permutation polynomial of $\mathbb F_{q^n}$ if and only if $\frac{M_t}{q-1}$ is even for every $t\in\mathbb F_{q^n}^*$, where $M_t$ is the number of roots of $\Tr\big(\lambda(L^\prime(x)^s\ell^\prime(tx)^q)^{q+1}\big)$ in $\mathbb F_{q^n}^*$ with $s=\sum_{i=1}^\frac{n-1}2q^{2i}$.
\end{proposition}
\begin{proof}
With the discussion above, for $t\in\mathbb F_{q^n}^*$ we have
\[\begin{split}q^nN_t&=\sum_{u\in\mathbb F_{q^n}^*}\sum_{w\in\mathbb F_{q^n}}\chi(L^\prime(u)w^{q+1}+\ell^\prime(tu)w)\\&=G\sum_{u\in\mathbb F_{q^n}^*}\eta(L^\prime(u))\chi\left(-\frac{\lambda\left(L^\prime(u)^s\ell^\prime(tu)^q\right)^{q+1}}{4\N(L^\prime(u))}\right).\end{split}\]
Let $R$ be a set of representatives from each coset of $\mathbb F_{q^n}^*/\mathbb F_q^*$, and $\psi$ be the canonical additive character of $\mathbb F_q$, so that
\[\begin{split}&\mathrel{\phantom{=}}\sum_{u\in\mathbb F_{q^n}^*}\eta(L^\prime(u))\chi\left(-\frac{\lambda\left(L^\prime(u)^s\ell^\prime(tu)^q\right)^{q+1}}{4\N(L^\prime(u))}\right)\\&=\sum_{r\in R}\sum_{u_0\in\mathbb F_q^*}\eta(u_0L^\prime(r))\chi\left(-\frac{u_0^{n+1}\lambda\left(L^\prime(r)^s\ell^\prime(tr)^q\right)^{q+1}}{4u_0^n\N(L^\prime(r))}\right)\\&=\sum_{r\in R}\eta(L^\prime(r))\sum_{u_0\in\mathbb F_q^*}\eta(u_0)\psi\left(-u_0\frac{\Tr\big(\lambda(L^\prime(r)^s\ell^\prime(tr)^q)^{q+1}\big)}{4\N(L^\prime(r))}\right)\\&=\sum_{r\in R}\eta(L^\prime(r))\eta\left(-\frac{\Tr\big(\lambda(L^\prime(r)^s\ell^\prime(tr)^q)^{q+1}\big)}{4\N(L^\prime(r))}\right)\sum_{u_0\in\mathbb F_q^*}\eta(u_0)\psi(u_0)\\&=\sum_{r\in R}\eta\big(-\Tr\big(\lambda(L^\prime(r)^s\ell^\prime(tr)^q)^{q+1}\big)\big)G_1,\end{split}\]
where $G_1$ is the quadratic Gauss sum of $\mathbb F_q$, since the restriction of $\eta$ to $\mathbb F_q$ is exactly the quadratic character of $\mathbb F_q$. Note that $G=(-1)^{n-1}G_1^n$ and $G_1^2=\eta(-1)q$, so
\[q^nN_t=(-1)^{n-1}G_1^{n+1}\sum_{r\in R}\eta\big(-\Tr\big(\lambda(L^\prime(r)^s\ell^\prime(tr)^q)^{q+1}\big)\big),\]
and the parity of $N_t$ is the same as that of
\[\sum_{r\in R}\eta\big(-\Tr\big(\lambda(L^\prime(r)^s\ell^\prime(tr)^q)^{q+1}\big)\big).\]
A simple investigation yields that $N_t$ is odd if and only if the number of $r\in R$ such that $\Tr\big(\lambda(L^\prime(r)^s\ell^\prime(tr)^q)^{q+1}\big)=0$ is even. On the other hand, the polynomial $\frac{L(x^{q+1})}{\ell(x)}$ permutes $\mathbb F_{q^n}$ if and only if $N_t$ is odd for every $t\in\mathbb F_{q^n}^*$. The proof is then completed.
\end{proof}

\begin{corollary}
Let $q$ and $n$ be odd, $s_0$ be an integer such that $\gcd(s_0,q^n-1)=\gcd(1+s,q^n-1)$ with $s=\sum_{i=1}^\frac{n-1}2q^{2i}$, and $h$ be a permutation polynomial of $\mathbb F_{q^n}$. If
\[x^s\ell^\prime(x)^q=\alpha h(x)^{s_0}+\beta h(x)^{s_0q^2}\]
for some $\alpha,\beta\in\mathbb F_{q^n}$ with $\N(\alpha)+\N(\beta)\ne0$, then $\frac{\ell(x)}{x^{q+1}}$ is a permutation polynomial of $\mathbb F_{q^n}$.
\end{corollary}
\begin{proof}
Note that $h(0)=0$ and for $t\in\mathbb F_{q^n}^*$,
\[(tx)^s\big(\alpha^{q^{n-1}}x^{q^{n-1}}+\beta^{q^{n-1}}x\big)^q=t^s(\alpha x^{1+s}+\beta x^{q+s})=t^s\big(\alpha x^{1+s}+\beta(x^{1+s})^{q^2}\big),\]
while
\[(tx)^s\ell^\prime(x)^q=t^s\big(\alpha h(x)^{s_0}+\beta h(x)^{s_0q^2}\big),\]
so the number of roots of
\[\Tr\big(\lambda\big((tx)^s\big(\alpha^{q^{n-1}}x^{q^{n-1}}+\beta^{q^{n-1}}x\big)^q\big)^{q+1}\big)\]
in $\mathbb F_{q^n}^*$ is equal to that of
\[\Tr\big(\lambda((tx)^s\ell^\prime(x)^q)^{q+1}\big).\]
Therefore, the fact that $\frac{x^{q+1}}{\alpha x^q+\beta^{q^{n-1}}x}$ (the transpose of $\alpha x^q+\beta^{q^{n-1}}x$ is exactly $\alpha^{q^{n-1}}x^{q^{n-1}}+\beta^{q^{n-1}}x$) is a permutation polynomial of $\mathbb F_{q^n}$ implies that so is $\frac{x^{q+1}}{\ell(x)}$, by the preceding proposition.
\end{proof}

\begin{example}\label{e1}
In the case $n=3$, consider $\ell(x)=(\N(a)+1)x^{q^2}+2a^qx^q+2a^{q+1}x$ with $a\in\mathbb F_{q^3}^*$ and $\N(a)^2\ne1$. Let $h(x)=a^{q^2+1}x^{q^2}+a^{q^2}x^q+x$, which is clearly a permutation polynomial of $\mathbb F_{q^3}$. Then
\[\begin{split}&\mathrel{\phantom{=}}a^{2q^2}h(x)^{2q}-h(x)^2\\&=a^{2q^2}\big(a^2x^{2q^2}+x^{2q}+a^{2q+2}x^2+2a^{q+2}x^{q^2+1}+2ax^{q^2+q}+2a^{q+1}x^{q+1}\big)\\&\mathrel{\phantom{=}-}\big(a^{2q^2+2}x^{2q^2}+a^{2q^2}x^{2q}+x^2+2a^{q^2+1}x^{q^2+1}+2a^{2q^2+1}x^{q^2+q}+2a^{q^2}x^{q+1}\big)\\&=(\N(a)^2-1)x^2+2\big(a^{2q^2+q+2}-a^{q^2+1}\big)x^{q^2+1}+2\big(a^{2q^2+q+1}-a^{q^2}\big)x^{q+1}\\&=(\N(a)-1)x\big(2a^{q^2+1}x^{q^2}+2a^{q^2}x^q+(\N(a)+1)x\big)\\&=(\N(a)-1)x\ell^\prime(x)^{q^2}.\end{split}\]
Hence, both
\[\frac{\ell(x)}{x^{q+1}}=(\N(a)+1)x^{q^2-q-1}+2a^qx^{-1}+2a^{q+1}x^{-q}\]
and
\[\frac{\ell^\prime(x^{q+1})}x=(\N(a)+1)x^{q^2+q-1}+2ax^{q^2}+2a^{q+1}x^q\]
are permutation polynomials of $\mathbb F_{q^3}$, and it can be verified that the inverse of $\ell$ is
\[\frac2{1-\N(a)^2}\left(a^{q^2+q}x^{q^2}+\frac{1-\N(a)}2x^q-a^{q^2}x\right),\]
not a binomial.
\end{example}

Assume now that $q=4^k$ for some integer $k$, and $n=3$. We show a permutation polynomial $\frac{L(x^{q+1})}{\ell(x)}$ with neither $L$ nor $\ell$ being a monomial. Let $\alpha$ be a primitive element of $\mathbb F_4$, and
\[L(x)=x^{q^2}+x^q+\alpha^2x\quad\text{and}\quad\ell(x)=\alpha x^{q^2}+x^q+x.\]
Note that $\gcd(q+1,q^3-1)=1$ with $\frac q2(q^2+q-1)(q+1)\equiv1\pmod{q^3-1}$, and
\[L(\ell(x))=x^{q^2}+\alpha x^q+x+x^{q^2}+x^q+\alpha x+\alpha^2\big(\alpha x^{q^2}+x^q+x\big)=x^{q^2}.\]
Then
\[\begin{split}L\big(L(x)^{q^2+q}\big)&=L\big(\big(\alpha^2x^{q^2}+x^q+x\big)\big(x^{q^2}+\alpha^2x^q+x\big)\big)\\&=L\big(\alpha^2x^{2q^2}+\alpha^2x^{2q}+x^2+\alpha x^{q^2+1}+\alpha^2x^{q^2+q}+\alpha x^{q+1}\big)\\&=\alpha^2L(\ell(x^{2q}))+\alpha L\big(\ell\big(x^{q^2+1}\big)\big)\\&=\alpha^2x^2+\alpha x^{q^2+q}.\end{split}\]
We have
\[\alpha^2x^2+\alpha x^{q^2+q}\]
as a composition of permutation polynomials, as well as
\[\alpha x+\alpha^2x^\frac{q^2+q}2.\]
Substituting $x^{q^2+q-1}$ for $x$, we get
\[\alpha x^{q^2+q-1}+\alpha^2 x=\frac{L\big(L(x)^{q^2+q}\big)}x,\]
and thus
\[\frac{L(x^{q+1})}{\ell(x)}=\frac{L\big(L(\ell(x))^{q^2+q}\big)}{\ell(x)}\]
is a permutation polynomial of $\mathbb F_{q^3}$.


\section{Conclusions}

We have obtained some permutation polynomials of $\mathbb F_{q^n}$ with inverses:
\begin{itemize}
\item $\frac{L(x^{q+1})}x$, where $L$ is the inverse of $x^q+ax$ for $a\in\mathbb F_{q^n}^*$ with $\N(-a)\ne1$;
\item $\frac{L(x)}{x^{q+1}}$, where $L$ is the inverse of $x^{q^{n-1}}+ax$ for $a\in\mathbb F_{q^n}^*$ with $\N(-a)\ne1$;
\item $bx^{q^{n-1}}+b^{q^k}x^{q^k+q^{k-1}-1}$ and $b^{q^k}x^{q^{k+1}+q^k-1}+bx^q$, where $n=2k$ and either $k$ is odd and $b$ is an element in $\mathbb F_{q^n}^*$ with $b^\frac{q^n-1}{q+1}\ne-1$, or $k$ is even, $q$ is odd, and $b$ is a square element in $\mathbb F_{q^n}^*$;
\item $a^{q^2}x^{q^3}+b^{q^2}x^{q^3+q^2-1}+ax^{q^2+q-1}+bx^q$, where $n=4$, $q$ is odd, $a$ and $b$ are square elements in $\mathbb F_{q^4}^*$ with $ab^{-q}\in\mathbb F_{q^2}$ and $\N(a)\ne\N(b)$.
\end{itemize}
Also, some other permutation polynomials can be derived from these. For future work, It may be interesting to explore more on permutation polynomials of the form $\frac{L(x^{q+1})}{\ell(x)}$.

\section*{Acknowledgement}

The work of the first author was supported by the China Scholarship Council. The funding corresponds to the scholarship for the PhD thesis of the first author in Paris, France. The French Agence Nationale de la Recherche partially supported the second author's work through ANR BARRACUDA (ANR-21-CE39-0009).


\end{document}